\documentclass[12pt]{amsart}

\usepackage[unicode]{hyperref}

\usepackage{amsmath,amsfonts,amssymb,amsthm}
\usepackage[matrix,arrow,curve]{xy}

\usepackage{longtable}

\makeatletter\@addtoreset{equation}{section} \makeatother

\newtheorem{theorem}[equation]{Theorem}
\newtheorem*{theorem*}{Theorem}
\newtheorem{proposition}[equation]{Proposition}
\newtheorem*{proposition*}{Proposition}

\newtheorem*{corollary*}{Corollary}
\newtheorem{conjecture}[equation]{Conjecture}

\theoremstyle{definition}

\theoremstyle{remark}
\newtheorem{remark}[equation]{Remark}
\makeatletter
\newcommand{\symbitem}[1]{\item[#1]%
\renewcommand{\@currentlabel}{#1}\ignorespaces}
\makeatother

\def \O {\mathcal{O}}

\def \C {\mathbb{C}}
\def \P {\mathbb{P}}
\def \Q {\mathbb{Q}}

\def \Z {\mathbb{Z}}

\def \geq {\geqslant}

\def \leq {\leqslant}

\def \kappa {\varkappa}

\newcommand{\udot}{{\:\raisebox{3pt}{\text{\circle*{1.5}}}}}
\def \bullet {\udot}

\DeclareMathOperator{\rk}{rk}

\DeclareMathOperator{\codim}{codim}

\DeclareMathOperator{\Pic}{Pic}

\author{Sergey Galkin}
\begin{title}
{Ap\'ery constants of homogeneous varieties.}
\end{title}
\begin{document}
\begin{abstract}
For Fano manifolds we define Ap\'ery constants and Ap\'ery class as particular limits of ratios of coefficients of solutions of the quantum differential equation. We do numerical computations in case of homogeneous varieties. These numbers are identified to be polynomials in the values $\zeta(k)$ of Riemann zeta-function with natural arguments $k$.
\end{abstract}
\maketitle



\setcounter{section}{0}

\section{Introduction}

The article is devoted to the computations of Ap\'ery numbers for the quantum differential equation
of homogeneous varieties, so first we introduce these $3$ notions.

Let $X$ be a Fano manifold of index $r$, that is $c_1(X) = r H$ for $H \in H^2(X,\Z)$,
and $q$ be a coordinate on the anti-canonical
torus $B := \Z c_1(X) \otimes \C^* = G_m \in \Pic(X) \otimes \C^*$,
and $D = q \frac{d}{dq}$ be an invariant vector field.
Cohomology $H^\bullet(X)$ are endowed with the structure of quantum multiplication $\star$,
and the first Dubrovin's connection on a trivial $H^\bullet(X)$-bundle over $B$ is given by
\begin{equation} \label{maineq}
D \phi = H \star \phi
\end{equation}


If we replace in equation \ref{maineq} quantum multiplication with the ordinary cup-product,
then its solutions are constant Lefschetz coprimitive (with respect to $H$) classes in $H^\bullet(X)$.
Dimension $\mu$ of the space of holomorphic solutions of \ref{maineq} is the same
and equal to the number of admissible
initial conditions (of the recursion on coefficients) modulo $q$,
i.e. the rank of the kernel of cup-multiplication by $H$ in
$H^\bullet(X)$,
that is the dimension of coprimitive Lefschetz cohomology.

Solving equation \ref{maineq} by Newton's method one obtains
 a matrix-valued few-step recursion reconstructing
all the holomorphic solutions from these initial conditions.

Givental's theorem states that the solution
$A = 1 + \sum_{n\geq 1} a^{(n)} q^n$ associated with the primitive class $1 \in H^0(X)$
is the $J$-series of the manifold $X$ (the generating function counting some rational curves of $X$).
Choose a basis of other solutions $A_1,\dots,A_{\mu-1}$ associated with homogeneous
primitive classes of nondecreasing codimension.
Put $A = \sum_{n\geq0} a^{(n)} t^n$ and $A_i = \sum_{n\geq0} a_i^{(n)} t^n$.
We call the number 
$$\lim_{n\to \infty} \frac{a_i^{(n)}}{a^{(n)}}$$
 {\it $i$-th Ap\'ery constant} after the renown work \cite{Apery},
where $\zeta(3)$ and $\zeta(2)$ were shown to be of that kind for some differential equations
 and such a presentation
was used for proving the irrationality of these two numbers.
If there is no chosen basis, for any coprimitive class $\gamma$ one still may consider
the solution $A_\gamma = \sum_{n\geq 1} a_\gamma^{(n)} q^n = 
Pr_0({\gamma + \sum_{n\geq1} A_\gamma^{(n)} q^n})$ and the limit 
\begin{equation}
Apery(\gamma) = \lim_{n\to \infty} \frac{a_\gamma^{(n)}}{a^{(n)}}
\end{equation}
Defined in that way, $Apery$ is a linear map from coprimitive cohomology to $\C$.
A linear map on coprimitive cohomology is dual
\footnote{One may choose between Poincare and Lefschetz dualities. We prefer the first one.}
to some (non-homogeneous) primitive cohomology class with coefficients in $\C$.
We name it {\it Ap\'ery characteristic class} $A(X) \in H^{\leq \dim X} (X,\C)$.

Consider the homogeneous ring $R = \Q[c_1,c_2,c_3,\dots], \deg c_i = i$
and a map $ev: R \to \C$ sending $c_1$ to Euler constant $C$
\footnote{$C = \lim_{n \to \infty} (\sum_{k=1}^n \frac{1}{k}) - \ln n$},
and $c_i$ to $\zeta(i)$.

The main conjecture we verify is the following
\begin{conjecture} \label{mainconj}
Let $X$ be any Fano manifold and $\gamma \in H^\bullet(X)$ be some coprimitive with respect to $-K_X$
homogeneous cohomology class of codimension $n$. Consider two solutions of
quantum $D$-module: $A_0$ associated with $1$ and $A_\gamma$ associated with $\gamma$.
Then Ap\'ery number for $A_\gamma$ (i.e. $\lim_{k\to \infty} \frac{a_\gamma^{(k)}}{a_0^{(k)}}$)
is equal to $ev(f_\gamma)$ for some homogeneous polynomial $f_\gamma  \in R^{(n)}$ of degree $n$.
\end{conjecture}

Actually, in our case there is no Euler constant contributions,
and the conjecture seems too strong to be true - it would imply that
some of differential equations studied in \cite{ASZ} has non-geometric origin
(at least come not from quantum cohomology), because their Ap\'ery numbers does not seem to be
of the kind described in the conjecture (e.g. Catalan's constant, $\pi^3$, $\pi^3 \sqrt{3}$).

From the other point of view, for toric varieties $X$ the solutions of quantum differential equations are known to be
pull-backs of hyper-geometric functions, coefficients of hyper-geometric functions
are rational functions of $\Gamma$-values, and the Taylor expansion
\begin{equation} \label{gammazeta}
\log \Gamma(1-x) = C x + \sum_{k\geq 2} \frac{\zeta(k)}{k} x^k
\end{equation}
suggests that all Ap\'ery constants would probably be rational functions in $C$ and $\zeta(k)$.
So the main conjecture \ref{mainconj} is at least as plausible as toric degeneration conjecture or hyper-geometric pull-back conjecture.
Also Ap\'ery limits like $\frac{91}{432} \zeta(3) - \frac{1}{216} \pi^3 \sqrt{3}$ may
appear as ``square roots'' or factors (convolutions with quadratic character)
 of geometric limits like $\frac{91^2}{432^2} \zeta(3)^2 - \frac{3}{216^2} \pi^6$.

This is not even the second paper (the computations of this paper were described to author by Golyshev in 2006) discussing the natural appearance of $\zeta$-values
in monodromy of quantum differential equations. In case of fourfolds $X$ the expression of monodromy
in terms of $\zeta(3), \zeta(2k)$ and characteristic numbers of anti-canonical section of $X$
was given by van Straten \cite{Straten},
$\Gamma$-class for toric varieties appears in Iritani's work \cite{Iritani},
and in general context in \cite{KKP}.

Let $G$ be a (semi)simple Lie group, 
$W$ be its Weyl group, 
$P$ be a (maximal) parabolic subgroup associated with
the subset (or just one) of the simple roots of Dynkin diagram,
 and denote factor $G/P$  by $X$.
$X$ is a homogeneous Fano manifold with $\rk \Pic X$ equal to the number
of chosen roots. In case when $G$ is simple and $P$ is maximal we have $\Pic X = \Z H$,
where $H$ is an ample generator, $K_X = -r H$.

For homogeneous varieties with small number of roots in Dynkin diagram (being more precise, with
not too big total dimension of cohomology)
by the virtue of Peterson's version of Quantum Chevalley formula \cite{FW04}[Theorem 10.1]
we explicitly compute the operator $H \star$
\footnote{We used computer algebra software LiE \cite{Lie} for the computations in Weyl groups.
The script is available at
\url{http://www.mi.ras.ru/~galkin/work/qch.lie},
and the answer is available in \cite{GG2}. We used PARI/GP computer algebra software \cite{PARI2}
for solving the recursion and finding the linear dependencies between the answers and zeta-polynomials.
Script for this routine is available at
\url{http://www.mi.ras.ru/~galkin/work/apery.gp}.
},
and hence find \ref{maineq}
with all its holomorphic solutions.
Then we do a numerical computation of the ratios 
$\frac{a_\gamma^{(k)}}{a_0^{(k)}}$ for big $k$ (e.g. $k=20$ or $40$ or $100$),
and guess the values of the corresponding Ap\'ery constants,
then state some conjectures (refining \ref{mainconj}) on what these numbers should be.

\section{Grassmannian Gr(2,N)}

Let $V$ be the tautological bundle on Grassmannian $Gr(2,N)$,
consider $H = c_1(V)$ and $c_2 = c_2(V)$. Cohomology $H^\bullet(Gr(2,N),\C)$
is a ring generated by $H$ and $c_2$ with relations of degree $\geq N-1$.
So there is at least $1$ primitive (with respect to $H$) Lefschetz cohomology class $p_{2k}$
in every even codimension $2k$, $0 \leq k \leq \frac{N-2}{2}$. Since
$$\dim H^\bullet(Gr(2,N),\C) = \binom{N}{2} = \sum_{k=0}^\frac{N-2}{2} (2N-3-4k)$$
they exhaust all the primitive classes.

\begin{gather*}
p_0 = 1 \\
p_2 = c_2 - \frac{c_2 \cdot c_1^{2N-6}}{c_1^{2N-4}} c_1^2 \\
\dots 
\end{gather*}

The associated conjectural Ap\'ery numbers are listed in the following table,
Ap\'ery numbers associated with the primitive cohomology classes of codimension $2k$
are rational multiples of $\zeta(2k) \simeq_{\Q^*} \pi^{2k}$.

\smallskip

\begin{center}
\begin{longtable}{|l|c|c|c|c|c|}
\hline
$X$ & $\mu$ & $p_2$ & $p_4$ & $p_6$ & $p_8$ \\
\hline \endhead
$Gr(2,4)$ & $2$ & $0$ & & & 			\\
$Gr(2,5)$ & $2$ & $\zeta(2)$ & & &		\\
\hline
$Gr(2,6)$ & $3$ & $2 \zeta(2)$ & $0$ & &		\\
$Gr(2,7)$ & $3$ & $3 \zeta(2)$ & $\frac{27}{4} \zeta(4)$	& & \\
\hline
$Gr(2,8)$ & $4$ & $4 \zeta(2)$ & $16 \zeta(4)$ & $0$ &			\\
$Gr(2,9)$ & $4$ & $5 \zeta(2)$ & $\frac{111}{4} \zeta(4)$ & $\frac{675}{16} \zeta(6)$ &		\\
\hline
$Gr(2,10)$ & $5$ & $6 \zeta(2)$ & $42 \zeta(4)$ & $108 \zeta(6)$	& $0$		\\
$Gr(2,11)$ & $5$ & $7 \zeta(2)$ & $\frac{235}{4} \zeta(4)$ & 
	$\frac{3229}{16} \zeta(6)$ & $\frac{18375}{64} \zeta(8)$ \\
\hline
$Gr(2,12)$ & $6$ & $8 \zeta(2)$ & $78 \zeta(4)$ & $328 \zeta(6)$ & $768 \zeta(8)$,
\\
$Gr(2,13)$ & $6$ & $9 \zeta(2)$ & $\frac{399}{4} \zeta(4)$ & $\frac{7855}{16} \zeta(6)$ &
 $\frac{96111}{64} \zeta(8)$,
 \\
\hline
$Gr(2,14)$ & $7$ & $10 \zeta(2)$ & $124 \zeta(4)$ & $695 \zeta(6)$ & $\frac{7664}{3} \zeta(8)$,
\\
$Gr(2,15)$ & $7$ & $11 \zeta(2)$ & $\frac{603}{4} \zeta(4)$ & $\frac{15113}{16} \zeta(6)$ &
$\frac{768085}{192} \zeta(8)$,
\\
\hline
\end{longtable}
\end{center}

\begin{remark}
$Gr(2,5)$ case is essentially Ap'ery's recursion for $\zeta(2)$ (see remark \ref{ql1}).
\end{remark}

\begin{remark}
Constants for $p_2$ depend linearly on $N$,
constants for $p_4$ depend quadratically on $N$,
constants for $p_6$ looks like they grow cubically in $N$.
So we conjecture constants for $p_{2k}$ is $\zeta(2k)$ times
polynomial of degree $k$ of $N$.
\end{remark}

The proof for the computation of $p_2$ (in slightly another $\Q$-basis) was given recently in \cite{GoZ}. 
Let us describe a transparent generalization of this method for the all primitive $p_{2k}$ of $Gr(2,N)$.
Quantum $D$-module for $Gr(r,N)$ is the $r$'th wedge power of quantum $D$-module for $\P^{N-1}$
(solutions of quantum differential equation for $Gr(r,N)$ are $r \times r$ Wronskians of the fundamental matrix 
of solutions for $\P^{N-1}$).
 Let $N$ be either $2n$ or $2n+1$. 
Consider the deformation
of quantum differential equation for $\P^{N-1}$:
\begin{equation} \label{defeq}
 (D-u_1)(D+u_1)(D-u_2)(D+u_2)\cdot \dots \cdot (D-u_n)(D+u_n) \cdot D^{N-2n} - q 
\end{equation}
This equation has (at least) $2n$ formal solutions:
$$ R_a =\sum_{k-a \in \Z_+}  \frac{1}{\Gamma(k-u_1)\Gamma(k+u_1)\cdot \dots \cdot \Gamma(k-u_n)\Gamma(k+u_n) 
\cdot \Gamma(k)^{N-2n}} q^k $$
for $a=u_1, -u_1, \dots, u_n, -u_n$.
Let $S_i = R_{u_i}' R_{-u_i} -  R_{-u_i}' R_{u_i}$ be the Wronskians.
Then $S_i = \sum_{k\geq0} s_i^{(k)} q^k$ for $i=1,\dots,n$
 are $n$ holomorphic solutions of the wedge square of the deformed equation \ref{defeq}.
Using his explicit calculation for the monodromy of hyper-geometric equation \ref{defeq}
and Dubrovin's theory, Golyshev computes the monodromy of $\wedge^2(\ref{defeq})$
and demonstrates {\it the sine formula}:
\begin{equation} \label{sinsin}
\lim_{k\to\infty} \frac{s_i^{(k)}}{s_j^{(k)}} = \frac{\sin (2 \pi u_i)}{\sin (2 \pi u_j)}
\end{equation}
So in the base of $S_1,\dots,S_n$ Ap\'ery numbers are $\frac{\sin (2 \pi u_i)}{\sin (2 \pi u_1)}$.
One then reconstructs the required Ap\'ery numbers by applying the inverse fundamental solutions
matrix to this sine vector, and limiting all $u_i$ to $0$.

\section{Other Grassmannians of type A}

Let $V$ be the tautological bundle on Grassmannian $Gr(3,N)$,
consider $H = c_1(V)$, $c_2 = c_2(V)$ and $c_3 = c_3(V)$.

Cohomology $H^\bullet(Gr(3,N),\C)$
are generated by $H$, $c_2$ and $c_3$ with relations of degree $\geq N-2$. In particular,
if $N>7$, then $1$, $c_2$, $c_3$, $c_2^2$ and $c_2 c_3$ generate
$H^{\leq 10}(X,\Q) = H^\bullet(X)/H^{>10}(X)$ as $\Q[c_1]$-module.
So there is $1$ primitive class in codimensions $0$,$2$,$3$,$4$ and $5$.

\begin{longtable}{|l|c|c|c|c|c|c|c|c|c|c|c|c|c|p{5cm}|}
\hline
$X$ & $\mu$ & $p_2$ & $p_3$ & $p_4$ & $p_5$ & $p_{\geq 6}$ \\
\hline \endhead
$Gr(3,6)$ &   $3$ & $0$ & $-6 \zeta(3)$ & & &							\\
$Gr(3,7)$ &   $4$ & $\zeta(2)$ & $-7 \zeta(3)$ & $-\frac{17}{4} \zeta(4)$ &  & 
$-\frac{49}{2} \zeta(3)^2 - \frac{945}{16} \zeta(6)$	\\
\hline
$Gr(3,8)$ &   $5$ & $2 \zeta(2)$ & $-8 \zeta(3)$ & $0$ & $-8\zeta(2)\zeta(3)-4\zeta(5)$ &
$-32 \zeta(3)^2 - 62 \zeta(6)$ 		\\
$Gr(3,9)$ &   $8$ & $3 \zeta(2)$ & $-9 \zeta(3)$ & $\frac{27}{4} \zeta(4)$ & 
$-\frac{27}{2} \zeta(2)\zeta(3) -\frac{9}{2} \zeta(5)$ & 
$\pm (\frac{81}{2} \zeta(3)^2 +\frac{871}{16} \zeta(6))$,
\dots
\\
$Gr(3,10)$ & $10$ & $4 \zeta(2)$ & $-10 \zeta(3)$ & $16 \zeta(4)$ & $-20 \zeta(2)\zeta(3) -5 \zeta(5)$ &

$\pm (50 \zeta(3)^2 + 32 \zeta(6))$,
\dots
		\\
$Gr(3,11)$ & $13$ & $5 \zeta(2)$ & $-11 \zeta(3)$ & $\frac{111}{4} \zeta(4)$ & 
$-\frac{55}{2} \zeta(2)\zeta(3) -\frac{11}{2} \zeta(5)$	&
$ (-\frac{121}{2} \zeta(3)^2 + \frac{110}{16} \zeta(6)) \pm \frac{45}{16} \zeta(6)$,
\dots
\\
\hline
\end{longtable}

\begin{remark}
One may notice that the Ap\'ery constants of $p_2$ and $p_4$ for $Gr(3,N)$
 are equal to the Ap\'ery constants of $p_2$, $p_4$ for $Gr(2,N-2)$.
Why? Is it possible to make an analogous statement for $p_6$ (obviously one should choose
another basis of two elements in $H^{12}(Gr(3,N))$ to vanish appearing $\zeta(3)^2$ terms)?
\end{remark}

\begin{remark}
$p_2$ is linear of $N$, $p_4$ is quadratic of $N$,
$p_3$ is linear of $N$, $p_5$ is quadratic of $N$.
\end{remark}

\begin{remark}
$p_5$ is quadratic polynomial of $N$ times $\zeta(2)\zeta(3)$
plus linear polynomial of $N$ times $\zeta(5)$.
Actually it is $-\frac{p_2 p_3 - N \zeta(5)}{2}$.
This gives a suggestion on a method of separating e.g. $\zeta(4)$
and $\zeta(2)^2$ in $p_4$ --- $\zeta(4)$ term should be only linear
and $\zeta(2)^2$ is quadratic in $N$. 
Similarly the coefficient at $\zeta(3)^2$ is quadratic in $N$
(and in the chosen basis $p_6$'th $\zeta(3)^2$-part is $\frac{p_3^2}{2}$).
\end{remark}

For $Gr(4,N)$ we still do have a unique primitive class of codimension $5$.

\begin{longtable}{|l|c|c|c|c|c|c|p{5cm}|}
\hline
$X$ & $\mu$ & $p_2$ & $p_3$ & $p_4$ & $p_4'$ & $p_5$ & $p_{\geq 6}$  \\
\hline \endhead
$Gr(4,8)$ & $8$ & $0$ & $-8 \zeta(3)$ & $-6 \zeta(4)$ & $0$ & none &
$32 \zeta(3)^2 + 50 \zeta(6)$ twice and $0_8$ \\
$Gr(4,9)$ & $12$ & $\zeta(2)$ & $-9 \zeta(3)$ & $\frac{21}{4} \zeta(4)$ & $\zeta(4)$ &
$-\frac{9}{2} (\zeta(2)\zeta(3) + \zeta(5))$ & 
$(\frac{81}{2} \zeta(3)^2 +\frac{117}{4} \zeta(6)) \pm \frac{159}{16} \zeta(6)$,
\dots
	\\	
$Gr(4,10)$ & $18$ & $2 \zeta(2)$ & $-10 \zeta(3)$ & $-2 \zeta(4)$ & $2 \zeta(4)$ & 
$-10 \zeta(2)\zeta(3) -5 \zeta(5)$ &
$50 \zeta(3)^2 + 31 \zeta(6)$,
$50 \zeta(3)^2$, $0_6$, \dots
	\\	
$Gr(4,11)$ & $24$ & $3 \zeta(2)$ & $-11 \zeta(3)$ & $\frac{15}{4} \zeta(4)$ & $3 \zeta(4)$ &
 $-\frac{33}{2} \zeta(2)\zeta(3) - \frac{11}{2} \zeta(5)$ &
$(\frac{121}{2} \zeta(3)^2 + \frac{35}{2} \zeta(6)) \pm \frac{197}{16} \zeta(6)$,
$\frac{27}{16} \zeta(6)$,\dots
\\
\hline
\end{longtable}

\begin{remark}
Ap\'ery of $p_3$ for $Gr(3,N)$ and $Gr(4,N)$ coincide.
Ap\'ery of $p_2$ for $Gr(4,N)$ is equal to Ap\'ery of $p_2$ for $Gr(3,N-2)$ and
Ap\'ery of $p_2$ for $Gr(2,N-4)$.
\end{remark}

For $Gr(5,10)$ we have $20$ Lefschetz blocks, they correspond to $20$ solutions,
and hence $19$ Ap\'ery constants. Some of them vanish, while some other coincide
(because solutions differ only by some character). 

\begin{longtable}{|l|c|c|c|c|c|c|c|c|}
\hline
$X$ & $\mu$ & $p_2$ & $p_3$ & $p_4$ & $p_4'$ & $p_5$ & $p_5'$ \\
\hline
$Gr(5,10)$ & $20$ & $0$ & $-10 \zeta(3)$ & $-6 \zeta(4)$ & $0$ & $10 \zeta(5)$ & $-10 \zeta(5)$ \\
$Gr(5,11)$ & $32$ & $\zeta(2)$ & $-11 \zeta(3)$ & $-\frac{21}{4} \zeta(4)$ &
 $\zeta(4)$ & $11 (\zeta(5)-\zeta(2)\zeta(3))$ & $-11 \zeta(5)$ \\
\hline
\end{longtable}

\section{B,C,D cases}

The picture for other $3$ series of classical groups is similar.

For $1\leq k\leq n$ let $D(n,k)$ denote homogeneous space of isotropic (with respect to
non-degenerate quadratic form) $k$-dimensional
linear spaces in $2n$-dimensional vector space. $D(n,k) = OGr(k,2n) = G/P$ where $G$ is $Spin(2n)$,
and maximal parabolic subgroup $P \subset G$ corresponds to $k$'th simple root counting from left to right.
Similarly define $B(n,k) = OGr(k,2n+1)$ and $C(n,k) = SGr(k,2n)$.

\begin{longtable}{|l|c|p{15cm}|}
\hline
$X$ & $\mu$ & Ap\'ery numbers \\ 
\hline \endhead
$B(3,2)$ & $2$ & $-2 \zeta(2)$.					\\
$B(4,2)$ & $3$ & $\zeta(2)$, $-\frac{41}{2} \zeta(4)$. 		\\
$B(4,3)$ & $3$ & $-4 \zeta(2)$, $-4 \zeta(3)$. 			\\
$B(4,4)$ & $2$ & $2 \zeta(3)$.					\\
\hline
$B(5,2)$ & $4$ & $3\zeta(2)$, $\frac{3}{2}\zeta(4)$ $-\frac{1191}{8}\zeta(6)$.  \\
$B(5,3)$ & $8$ &
$0_2$, 
$-8\zeta(3)$, 
$-24\zeta(4)$, 
$ 20\zeta(5)$, 
$\frac{64}{3}\zeta(3)^2 + \frac{80}{3}\zeta(6)$,
$ 32\zeta(3)\zeta(4) + \frac{232}{3}\zeta(7)$,
$ \frac{256}{21}\zeta(3)^3 + \frac{320}{7}\zeta(3)\zeta(6) -
 \frac{480}{7}\zeta(4)\zeta(5) -\frac{1000}{21}\zeta(9)$.  
 \\
$B(5,4)$ & $8$ &
$-6\zeta(2)$, 
$-6\zeta(3)$, 
$-45\zeta(4)$, 
$ 9\zeta(2)\zeta(3) + 21\zeta(5)$, 
$ 15\zeta(3)^2 + \frac{1141}{24}\zeta(6)$, 
$ 56\zeta(2)\zeta(5) + 30\zeta(3)\zeta(4) + 52\zeta(7)$, 
$ \frac{266}{5}\zeta(3)^3 -\frac{171}{5}\zeta(2)\zeta(7) 
-\frac{222}{5}\zeta(3)\zeta(6) - \frac{263}{5}\zeta(4)\zeta(5) + \frac{136}{5}\zeta(9)$. 
 \\
$B(5,5)$ & $3$ &  $4 \zeta(3)$, $20 \zeta(5)$.			\\
\hline
$B(6,2)$ & $5$ & $5\zeta(2)$, $\frac{87}{4}\zeta(4)$,
 $-\frac{485}{8}\zeta(6)$, $-\frac{35073}{32}\zeta(8)$.  \\
$B(6,3)$ & $12$ & 
$2\zeta(2)$, $-6\zeta(3)$, $-12\zeta(4)$,
$-12\zeta(2)\zeta(3) + 18\zeta(5)$, 
$-36\zeta(3)^2 - 146\zeta(6)$, 
$ 36\zeta(3)^2 + 2\zeta(6)$, 
$ 24\zeta(2)\zeta(5) + 24\zeta(3)\zeta(4) + 76\zeta(7)$, 
$ \frac{360\zeta(3)^2\zeta(2) -1080\zeta(3)\zeta(5) + 1176\zeta(8)}{11}$, 
$ 803\zeta(3)^3 -528\zeta(2)\zeta(7) + 318\zeta(3)\zeta(6) -244\zeta(4)\zeta(5) -35\zeta(9)$, 
$ 75\zeta(3)^3 -336\zeta(2)\zeta(7) -395\zeta(3)\zeta(6) -22\zeta(4)\zeta(5) -70\zeta(9)$,\dots 
\\
$B(6,4)$ & $18$ & 
$-1 \zeta(2)$,
$-10 \zeta(3)$,
$-\frac{17}{4} \zeta(4)$,
$-14 \zeta(4)$,
$5 \zeta(2)\zeta(3)+ 19 \zeta(5)$,
$50 \zeta(3)^2+ 317 \zeta(6)$,
$-50 \zeta(3)^2- \frac{4135}{8} \zeta(6)$,
\\
$B(6,5)$ & $14$ & 
$-8 \zeta(2)$,
$-8 \zeta(3)$,
$-84 \zeta(4)$,
$64 \zeta(2)\zeta(3) + 16 \zeta(5)$,
$-64 \zeta(2)\zeta(3)$,
$\frac{80}{3} \zeta(3)^2 + 24 \zeta(6)$,
$110 \zeta(2)\zeta(5) + \frac{49}{2} \zeta(3)\zeta(4) + \frac{101}{2} \zeta(7)$,

\\
$B(6,6)$ & $5$ & $6 \zeta(3)$, $18 \zeta(5)$, $-18 \zeta(3)^2 - 60 \zeta(6)$, 
$36 \zeta(3)^3 +  360 \zeta(3)\zeta(6) + 332 \zeta(9)$
 \\
\hline
$B(7,2)$ & $6$ & $7 \zeta(2)$, $\frac{211}{4} \zeta(4)$, $\frac{1733}{8} \zeta(6)$,
$-\frac{76699}{96} \zeta(8)$,
$-\frac{5368203}{640} \zeta(10)$.
\\
$B(7,7)$ & $8$ & $8 \zeta(3)$, $16 \zeta(5)$, $-30 \zeta(3)^2 - 60\zeta(6)$, $-112 \zeta(7)$,
$\frac{256}{3} \zeta(3)^3 + 480 \zeta(3)\zeta(6) + \frac{992}{3}\zeta(9)$, \dots   \\
\hline
\end{longtable}

\begin{remark}
$B(4,4)$ case is essentially Ap\'ery's recursion for $\zeta(3)$.
\end{remark}

\begin{longtable}{|l|c|p{15cm}|}
\hline
$X$ & $\mu$ & Ap\'ery numbers  \\ 
\hline \endhead
$C(3,2)$ & $2$ & $2 \zeta(2)$.						\\
$C(3,3)$ & $2$ & $\frac{7}{2} \zeta(3)$.				\\
$C(4,2)$ & $3$ & $4 \zeta(2)$, $16 \zeta(4)$.				\\
$C(4,3)$ & $4$ & $\zeta(2)$, $-9 \zeta(3)$, 
			$-\frac{9}{2} (\zeta(2)\zeta(3) + \zeta(5))$.	\\
$C(4,4)$ & $2$ & $4 \zeta(3)$.						\\
\hline
$C(5,2)$ & $4$ & $6 \zeta(2)$, $42 \zeta(4)$, $108 \zeta(6)$.		\\
$C(5,3)$ & $8$ &
$3\zeta(2)$,
$-11\zeta(3)$,
$\frac{27}{4}\zeta(4)$,
$-\frac{33}{2}\zeta(2)\zeta(3) - \frac{11}{2}\zeta(5)$,
$\frac{242}{3}\zeta(3)^2 + \frac{2383}{48}\zeta(6)$,
$-11\zeta(2)\zeta(5) -\frac{99}{4}\zeta(3)\zeta(4) - \frac{11}{3}\zeta(7)$,
$108\zeta(3)^3 -38\zeta(2)\zeta(7) + \frac{309}{4}\zeta(3)\zeta(6) - \frac{41}{4}\zeta(4)\zeta(5) + 36\zeta(9) $
						\\
$C(5,4)$ & $8$ & 
$0_2$,
$-10\zeta(3)$,
$30\zeta(4)$,
$-5\zeta(5)$,
$\frac{250}{3}\zeta(3)^2 + \frac{175}{3}\zeta(6)$,
$-\frac{100}{3}\zeta(3)\zeta(4) - \frac{10}{9}\zeta(7)$,
$\frac{2500}{21}\zeta(3)^3 + 250\zeta(3)\zeta(6) - \frac{150}{7}\zeta(4)\zeta(5) - \frac{10}{21}\zeta(9)$.
				\\
$C(5,5)$ & $3$ & $\frac{9}{2} \zeta(3)$ ,  $-\frac{21}{2} \zeta(5)$.	\\
\hline
$C(6,2)$ & $5$ & $8 \zeta(2)$, $78 \zeta(4)$, $328 \zeta(6)$, $768 \zeta(8)$.  \\
$C(6,3)$ & $12$ &
$5\zeta(2)$,
$-13\zeta(3)$,
$\frac{111}{4}\zeta(4)$,
$-\frac{65}{2}\zeta(2)\zeta(3) - \frac{13}{2}\zeta(5)$,
$-\frac{169}{2}\zeta(3)^2 + \frac{155}{16}\zeta(6)$,
$\frac{169}{2}\zeta(3)^2 + \frac{65}{2}\zeta(6)$,
 							\\
$C(6,6)$ & $4$ &
$\zeta(3)$,
$-11 \zeta(5)$,
$-25 \zeta(3)^2 - \frac{15}{2} \zeta(6)$,
$\frac{500}{3} \zeta(3)^3 + 150 \zeta(3)\zeta(6) - \frac{131}{3}\zeta(9)$.
 							\\
\hline
$C(7,2)$ & $6$ & $10 \zeta(2)$, $124 \zeta(4)$, $695 \zeta(6)$,
 $\frac{7664}{3} \zeta(8)$, $5760 \zeta(10)$.  				\\
$C(7,7)$ & $8$ & $\frac{11}{2}\zeta(3)$, $-\frac{23}{2}\zeta(5)$, 
$-\frac{121}{4}\zeta(3)^2 -\frac{15}{2}\zeta(6)$, 
$\frac{71}{2}\zeta(7)$, 
$\frac{1331}{6}\zeta(3)^3 + 165\zeta(3)\zeta(6) -\frac{263}{6}\zeta(9)$, 
$\frac{781}{12}\zeta(3)\zeta(7) - \frac{529}{12}\zeta(5)^2 - \frac{63}{2}\zeta(10)$,\dots 
 							\\
\hline
\end{longtable}

\begin{remark}
One may notice that Ap\'ery numbers for $C(2,n)=SGr(2,2n)$ coincide with Ap\'ery numbers
of $Gr(2,2n)$ except the last $0$. The reason for this coincidence is that $SGr(2,2n)$
is a hyperplane section of $Gr(2,2n)$, so by quantum Lefschetz (see \ref{ql1})
it has almost the same Ap\'ery numbers.
\end{remark}

\begin{remark}
For general $k$ spaces $OGr(k,N)$ and $SGr(k,N)$ are sections of ample vector bundles
over $Gr(k,N)$ (symmetric and wedge square of tautological bundle). Is it possible
to formulate a generalization of quantum Lefschetz principle explaining the relations
between Ap\'ery numbers of $OGr(k,N)$, $SGr(k,N)$ and $Gr(k,N)$?
\end{remark}

\begin{longtable}{|l|c|p{15cm}|}
\hline
$X$ & $\mu$ & Ap\'ery numbers  \\
\hline \endhead
$D(4,2)$ & $4$ &  $0, 0, -24 \zeta(4)$.				\\
\hline
$D(5,2)$ & $5$ & $2 \zeta(2)$, $0$, $-12 \zeta(4)$, $-144 \zeta(6)$. \\

$D(5,3)$ & $9$ & $- \zeta(2)$, $-\zeta(2)$, $-6 \zeta(3)$, $0_4$, $-\frac{45}{2} \zeta(4)$,
$3 \zeta(2)\zeta(3)+ 21 \zeta(5)$, $0_5$, $12 \zeta(3)^2 + \frac{275}{24} \zeta(6)$.
  						\\
$D(5,4)$ & $2$ & $2 \zeta(3)$.					\\
\hline
$D(6,2)$ & $6$ & $4 \zeta(2)$, $10 \zeta(4)$, $10 \zeta(4)$, $-124 \zeta(6)$, $-960 \zeta(8)$. \\
$D(6,3)$ & $14$ &
$ \zeta(2)$,
$-5 \zeta(3)$,
$-5 \zeta(3)$,
$-\frac{41}{2} \zeta(4)$,
$0$,
$-5 \zeta(2)\zeta(3)+ 19 \zeta(5)$,
$\frac{25}{2} \zeta(3)^2 + \frac{953}{16} \zeta(6)$,
$\frac{25}{2} \zeta(3)^2 - \frac{937}{16} \zeta(6)$,
$0$,
 							\\
$D(6,5)$ & $3$ & $4 \zeta(3)$, $20 \zeta(5)$.			\\		
\hline
$D(7,2)$ & $7$ & $6 \zeta(2)$, $36 \zeta(4)$, $0$, $50 \zeta(6)$, $-1072 \zeta(8)$, $-6912 \zeta(10)$.  \\
$D(7,6)$ & $5$ & $6 \zeta(3)$, $18 \zeta(5)$,
$-18 \zeta(3)^2 - 60 \zeta(6)$, 
$36\zeta(3)^3 + 360\zeta(3)\zeta(6) + 332\zeta(9)$.  			\\
\hline
\end{longtable}

\begin{remark}
$D(N,N-1)$ is isomorphic to $B(N-1,N-1)$, so in the case $D(6,5)$
we again have Ap\'ery's recurrence for $\zeta(3)$ here.
\end{remark}

\section{Exceptional cases - $E$, $F$, $G$}
We provide computations of Ap\'ery constants
only for a few of $23$ exceptional homogeneous varieties, those with not too big spaces
of cohomology.

\begin{longtable}{|l|c|p{15cm}|}
\hline
$X$ & $\mu$ & Ap\'ery numbers  \\
\hline
$E(6,6)$ & $3$ & $6 \zeta(4)$, $0_8$.			\\
$E(6,2)$ & $6$ & $0_3$, $18\zeta(4)$, $90\zeta(6)$, $0_7$, $-3456 \zeta(10)$. \\

$E(7,7)$ & $3$ & $-24 \zeta(5)$, $168 \zeta(9)$.	\\
$E(8,8)$ & $11$ & $120 \zeta(6)$, $-1512 \zeta(10)$,
 \dots (of degrees $12$, $16$, $18$, $22$, $28$). 		\\
\hline
$F(4,1)$ & $2$ & $21 \zeta(4)$.				\\
$F(4,3)$ & $8$ & $-4\zeta(2)$, $0_3$, $-2\zeta(4)$, $-24\zeta(5)$, 
$-246\zeta(6)$, $ 32\zeta(2)\zeta(5) + 60\zeta(7)$, 
$ 2160\zeta(2)\zeta(7) - 144\zeta(4)\zeta(5)$. \\ 
$F(4,4)$ & $2$ & $6 \zeta(4)$.				\\
\hline
\end{longtable}


\begin{remark}
There are two roots in the root system of $G_2$, taking factor by the parabolic subgroup
associated with the smaller one we get a projective space, so later by $G_2/P$ we
denote the $5$-dimensional factor by another maximal parabolic subgroup.
There is no literal Ap\'ery constants for $G_2/P$ since this variety is minimal,
so the only primitive cohomology class is $1$,
altough one may seek for almost solutions of quantum differential equation
(strictly speaking Ap\'ery himself also considered such solutions).
In \cite{GoZ} Golyshev considers this problem for Fano threefold $V_{18}$
(i.e. a section of $G_2/P$ by two hyperplanes) and using Beukers argument \cite{Beukers87}
and modularity of the quantum $D$-module for $V_{18}$ shows that Ap\'ery number is equal to
$L_{\sqrt{-3}}(3)$
\end{remark}

\section{Varieties with greater rank of Picard group, non-Calabi-Yau and Euler constant}
\label{euler}

One may consider the same question for varieties $X$ with higher Picard group.
Canonically we should put $H = -K_X$, but if we like, we could choose any $H \in \Pic(X)$.

Even for such simple spaces as products of projective spaces one immediately
calculates some non-trivial Ap\'ery constants.

\begin{longtable}{|l|c|c|}
\hline
$X$ & $\mu$ & Ap\'ery numbers  \\
\hline
$\P^2 \times \P^2$ & $3$ & $0_1$, $6 \zeta(2)$.                   \\
$\P^2 \times \P^3$ & $3$ & $0_1$, $\frac{14}{3} \zeta(2)$. \\
\hline
\end{longtable}

In all these cases Ap\'ery numbers corresponding to all primitive divisors vanish.
Van Straten's calculation \cite{Straten} relates monodromy of quantum differential equation for Fano fourfold $X$
not to Chern numbers of the Fano, but to Chern numbers of its anti-canonical Calabi-Yau
hyperplane section $Y$. Probably $C$-factors should correspond to $c_1$-factors
in the Chern number, and since for Calabi-Yau $c_1(Y)=0$ we observe Euler constant is
not involved. So one should consider something non-anti-canonical.

Let's test the case $H=\O(1,1)$ on $\P^2 \times \P^3$.
Being exact, we restrict $D$-module to sub-torus corresponding to $H$,
and consider operator of quantum multiplication by $H$ on it (sub-torus
associated with $H$ is invariant with respect to vector field associated with $H$).

\begin{longtable}{|l|c|c|}
\hline
$(X,H)$ & $\mu$ & Ap\'ery numbers  \\
\hline
$(\P^2 \times \P^3,\O(1,1))$ & $3$ & $-C$, $\frac{C^2 + 7 \zeta(2)}{2}$. \\
\hline
\end{longtable}

\section{Irrationality, special varieties and further speculations}
First of all let us note that both differential equations considered by Ap\'ery for the
proofs of irrationality of $\zeta(2)$ and $\zeta(3)$ are essentially appeared in our
computations as quantum differential equations of homogeneous varieties $Gr(2,5)$ and
$OGr(5,10) = D(5,4)$ (and isomorphic $OGr(4,9)= B(4,4)$).
By essentially we mean the following proposition --- Ap\'ery constants are invariant with respect
to taking hyperplane section if the corresponding primitive classes survive:
\begin{proposition} \label{ql1}
Let $X$ be a sub-canonically embedded smooth Fano variety\footnote{One may state this proposition
in higher generality, but we are going to use it for homogeneous spaces, and as stated it will
be enough.} of index $r>1$ i.e. $X$
is embedded to the projective space by a linear system $|H|$,
and $-K_X = rH$.
Consider a general hyperplane section $Y$ --- a sub-canonically embedded smooth Fano variety
of index $r-1$. There is a restriction map $\gamma \to \gamma \cap H$ from cohomology
of $X$ to cohomology of $Y$ and by Hard Lefschetz theorem except possible of intermediate codimension
all primitive classes of $Y$ are restricted primitive classes of $X$. Consider a homogeneous
primitive class of non-intermediate codimension $\gamma \in H^\bullet(X)$. Then Ap\'ery numbers
for $\gamma$ calculated from quantum differential equations of $X$ and $Y$ coincide.
\end{proposition}
\begin{proof}
By the quantum Lefschetz theorem of Givental-Kim-Gathmann we have a relation between the I-series
(solution of \ref{maineq} associated with $1 \in H^\bullet$) of $X$ and $Y$: e.g.
if $r>2$ and $\Pic(X) = \Z H$ and $H_2(X,\Z) = \Z \beta$
 then $d'th$ coefficient of $I-series$ of $X$
should be multiplied by $\prod_{i=0}^{d H \beta} (H+i)$, if $r\leq2$ one should also do a change
of coordinate.
One may show the similar relation between solutions of \ref{maineq} associated with $\gamma$
and $\gamma_{|_Y}$: either directly repeating the arguments of original proof, or by Frobenius
method of solving differential equation.
So the limit of the ratio is the same.
\end{proof}
One may rephrase the previous proposition in the following way
\begin{proposition} \label{ql2}
Ap\'ery class is functorial with respect to hyperplane sections.
\end{proposition}
Proposition \ref{ql2} is slightly stronger then \ref{ql1}:
indeed, the intermediate primitive classes of $X$ vanish restricted on $Y$,
but also it states that "parasitic" intermediate primitive classes of $Y$
has Ap\'ery constant equal to $0$.
Following notations of \cite{GoS} let's call all smooth varieties related to each other
by hyperplane section or deformation {\it a strain}, and if $Y$ is a hyperplane section of $X$
let's call $X$ {\it an unsection} of $Y$; if $Y$ has no unsections we call it
{\it a progenitor} of the strain.
The stability of Ap\'ery class is quite of the same nature as the stability of spectra
in the strain described in \cite{GoS}.
Propositions \ref{ql1} and \ref{ql2} suggest to consider some kind of stable Ap\'ery class
on the infinite hyperplane unsection. Such a stable framework of Gromov--Witten invariants
was constructed by Przyjalkowki for the case of quantum minimal Fano varieties in \cite{Prz},
using only Kontsevich-Manin axioms. The next proposition shows that
literally this construction gives nothing from our perspective 

\begin{proposition} \label{qmprop}
If a Fano manifold $X$ is quantum minimal then all Ap\'ery constants vanish i.e.
Ap\'ery class $A$ is equal to $1$.
\end{proposition}
\begin{proof}
It is a trivial consequence of the definition of quantum minimality ---
since all primitive classes except $1$ are quantum orthogonal
to $\C[K_X]$ the operator of quantum multiplication by $K_X$ restricted to
non-maximal Lefschetz blocks coincides with the cup-product, in particular it is
nilpotent, so the associated solutions $A_\gamma$ of quantum differential equation are polynomial in $q$ i.e.
their coefficients $a_\gamma^{(k)}$ vanish for $k>>0$, hence the Ap\'ery number is $0$.
\end{proof}
\begin{conjecture}
The converse to \ref{qmprop} statement is true as well.
\end{conjecture}
So for our purposes the framework of \cite{Prz} should be generalized taking into
account the structure of Lefschetz decomposition.
Another obstacle is geometrical nonliftability of varieties to higher dimensions --- one can show
both Grassmannian $Gr(2,5)$ (and any other Grassmannian except projective spaces and quadrics)
 and $OGr(5,10)$ are progenitors of their strains,
i.e. cannot be represented as a hyperplane section of any nonsingular manifold,
this follows e.g. from the fact that these varieties are self-dual,
but of course they are hyperplane sections of their cones.
\smallskip
We insist that the quantum recursions for the progenitors
$Gr(2,5)$ and $OGr(5,10)$ are the most natural in the strain,
in particular in both cases we consider two exact solutions of the recursion,
and in Ap\'ery's case one considers an almost solution with polynomial error term ---
because for the linear sections of dimension $\leq3$ ($\leq 5$) the second Lefschetz block
vanishes.
\bigskip 
One may ask a natural question whether any of the experimentally or theoretically calculated
Ap\'ery numbers (and their representations as the limits of the ratios of coefficients
of two solutions of the recurrence) may be proven to be irrational by Ap\'ery's argument.
At least we know it works in two cases of $Gr(2,5)$ and $OGr(5,10)$. 
Remind that for irrationality of $\alpha = \zeta(2)$ or $\alpha=\zeta(3)$ one shows 
that $(\alpha - \frac{a_\gamma}{q_n})$ is smaller then $\frac{1}{q_n}$, so we are interested
in the sign of $\lim \log(|\alpha - \frac{a_\gamma}{q_n}|) - \log(q_n)$
(or equivalently in the sign of 
\begin{equation} \label{convsp}
\lim \log \log(|\alpha - \frac{a_\gamma}{q_n}|) - \log \log q_n .
\end{equation}
There were many attempts to find any other recurencies with this sign being negative,
and most of them failed to the best of our knowledge. The quantum recursions we
considered in this article is not an exception (we calculated convergence speed \ref{convsp}
numerically for $n\geq 20$). For example the convergence speed for $\zeta(2)$
approximation from $Gr(2,N)$ decreases as $N$ grows, and is suitable only in the case of $Gr(2,5)$.
So we come to the question: what is so special about $Gr(2,5)$ and $OGr(5,10)$?
One immediately reminds the famous theorem of Ein (see e.g. \cite{Zak})
\begin{theorem} \label{einth}
Let $X \subset \P^N$ be a smooth non-degenerate irreducible $n$-dimensional
variety, such that $X$ has the same dimension as its projectively dual $X^*$.
Assume $N\geq \frac{3n}{2}$.
Then $X$ is either a hypersurface, or one of
\begin{enumerate} \label{sp1}
\item a Segre variety $\P^1 \times \P^r \subset \P^{2r+1}$
\item the Plucker embedding $Gr(2,5) \subset \P^9$
\item $OGr(5,10)$
\end{enumerate}
Three last cases are self-dual: $X \simeq X^*$.
\end{theorem}
\begin{remark}
For \ref{sp1} we have the coincidence of the coherent and topological cohomology
\begin{equation} \label{topcoh} 
N+1 = \dim H^0(X,\O(H)) = \dim H^\bullet(X)
\end{equation}
In all $3$ cases there are exactly two Lefschetz blocks, the codimensions of the grading
of second Lefschetz block are corr. $1$, $2$ and $3$.
\end{remark}
\begin{remark}
Ap\'ery number for $\P^1 \times \P^r$ should approximate some multiple
of $C$, but for $r=1,2,3$ it is $0$. As pointed out in section \ref{euler}
we haven't got any natural approximations for Euler constant in anti-canonical
Landau--Ginzburg model.
From the other point of view, the variety $\P^1 \times \P^r$ in the statement
of the theorem \ref{einth} is not (sub)anti-canonically embedded, but
embedded by the linear system $\O(1,1)$.
Calculations of \ref{euler} are what we expect to be
the quantum recursion for $X$ embedded by $\O(1,1)$,
they indeed approximate $C$, but the speed of convergence is too slow.
Either our guess is not correct (or not working here) or Landau--Ginzburg corresponding
to the linear system $|\O(1,1)|$ is something else.
\end{remark}
So the theorem \ref{einth} suggests the irrationality of Ap\'ery approximations are ruled by either
self-duality or extremal defectiveness of the progenitor. Varieties \ref{sp1}
are related by the famous construction: let $X$ be one of them, choose any point $p \in X$
(they are homogeneous so all points are equivalent),
then take an intersection of $X$ with its tangent space $Y = X \cap T_p X$. Then $Y$ is a cone
over the previous one:
\begin{gather}
T_p Gr(2,5) \cap Gr(2,5) = Cone(\P^1 \times \P^2) \\
T_p OGr(5,10) \cap OGr(5,10) = Cone(Gr(2,5))
\end{gather}

In that way $OGr(5,10)$ can be "lifted" one step further to Cartan variety $E(6,6)=E(6,1)$:
$$ T_p E(6,6) \cap E(6,6) = Cone(OGr(5,10)).$$
$E(6,6)$ is one of the four famous Severi varieties (or more general class of Scorza varieties)
classified by Fyodor Zak in \cite{Zak}:
\begin{theorem}
Let $X \subset \P^{N = \frac{3n+4}{2}}$ be $n$-dimensional Severi variety i.e.
$X$ can be isomorphically projected to $\P^{N-1}$.
Then $X$ is projectively equivalent to one of
\begin{enumerate}
\item the Veronese surface $v_2(\P^2) \subset \P^5$
\item the Segre fourfold $\P^2 \times \P^2 \subset \P^8$
\item the Grassmannian $Gr(2,6) \subset \P^{14}$
\item the Cartan variety $E(6,6) \subset \P^{26}$
\end{enumerate}
\end{theorem}
\begin{remark}
Apart from the first case that should be correctly interpreted
(e.g. taking symmetric square of $D$-module for $\P^2$), in the
other $3$ cases coincidence \ref{topcoh} holds
(this is general fact for the closures of highest weight orbits of algebraic groups).
The Lefschetz decompositions now consist of $3$ blocks --- first associated with $1$,
next one, and one block of length $1$ in intermediate codimension.
The last block has Ap\'ery number equal to $0$.
\end{remark}
Neither of Severi varieties provides us with a fast enough approximation, but the speeds
of convergence for them seem to be better then for arbitrary varieties.
So it may be possible that these speeds are related with the defect of the variety
(it is also supported by the fact that for Grassmannians defect decreases when $N$ grows). 

\bigskip
From the other perspective, when there are more then two Lefschetz blocks in the decomposition
one may try to use the simultaneous Ap\'ery-type approximations of a tuple of zeta-polynomials
as in the works of Zudilin. 

We would like to note that the recursion \ref{maineq} contains more then one approximation
of every Ap\'ery number appearing. Clearly speaking, in the definition
of Ap\'ery numbers we considered the limit of the ratios of fundamental terms i.e.
projections of two solutions $A_0$ and $A_\gamma$ to $H^0(X)$. It is natural to
ask if we get anything from considering the limits of ratios of the other coordinates.
Our experiments support the following 
\begin{conjecture}
$A_\gamma^{(k)}$ is approximately equal to $Apery(\gamma) \cdot A_0^{(k)}$ as $k\to\infty$.
\end{conjecture}
One may divide $A_\gamma^{(k)}$ by $A_0^{(k)}$ in the nilpotent ring of $H^\bullet(X)$ and state the
limit of such ratio exists and is equal to $Apery(\gamma) \in H^0(X)$.
For homogeneous $\gamma_2$ the ratio $\frac{(A_\gamma^{(k)},\gamma_2)}{(A_\gamma^{(k)},1)}$
grows as $k^{\codim \gamma_2}$ and the coordinates in the same Lefschetz block are linearly
dependent.

Finally let us provide some speculations explaining why the described behaviour is natural
and also why zeta-values should appear.
Assume for simplicity that 
the matrix of quantum multiplication by $H$ has degree $1$ in $q$ (it is often the case
for homogeneous varieties). Let $M_0$ be the operator of cup-product by $H$
and $M_1$ be the degree $1$ coefficient of quantum product by $H$.
Then the quantum recursion is one-step:
\begin{equation}
A^{(n)} = \frac{1}{n-M_0} M_1 A^{n-1} = \frac{1}{n} (1 + \frac{M_0}{n} + \frac{M_0^2}{n^2} + \dots) 
\cdot M_1 A^{n-1}
\end{equation}
Assume $M_0$ and $M_1$ commutes (actually, this is never true in our case).
Then 
$$A^{(l)} = 
\frac{1}{l!} \prod_{n=1}^l (1 + \frac{M_0}{n} + \frac{M_0^2}{n^2} + \dots) \cdot M_1^l A^{(0)}$$
Put

$$ N_l = \prod_{n=1}^l (1 + \frac{M_0}{n} + \frac{M_0^2}{n^2} + \dots) = 
\exp(\sum_{n=1}^l \sum_{k\geq 1} \frac{1}{k} \frac{M_0^k}{n^k}). $$ 

Up to normalization $\lim N_l$ is $\Gamma(1+M_0)$.
Assume further that largest (by absolute value) eigenvalue $\alpha$ of $M_1$ has the unique
eigenvector $\beta$ of multiplicity $1$.
Then $A^{(l)}$ is approximately equal to
$$C(A^{(0)},\beta) \cdot \frac{1}{l!} \cdot \alpha^l \cdot N_l \beta $$

Since $M_0$ and $M_1$ doesn't commute there are additional terms from the commutators of
$\Gamma(1 + M_0)$ and $M_1$.

\bigskip

{\bf Acknowledgement.} This preprint was written in late 2008.
The database \cite{GG2} of quantum cohomology of homogeneous varieties
was prepared jointly with V.\,Golyshev during our interest in the spectra
of Fano varieties (inspired by \cite{GG}), and the interest to the Ap\'ery constant appeared
after V.\,Golyshev's deresonance computation \cite{GoZ} of the first Ap\'ery constant for $Gr(2,N)$. 

Author thanks Duco van Straten for the fruitful discussions and the
invitation to the SFB/TR 45 program at Johannes Gutenberg Universit\"at.


\end{document}